\documentclass[12pt,reqno]{amsart}

\usepackage{verbatim}

\usepackage{amsthm,amssymb,amstext,amscd,amsfonts,amsbsy,amsxtra,latexsym,amsmath,xcolor,mathrsfs,fancybox}
\usepackage[english]{babel}
\usepackage[latin1]{inputenc}
\usepackage{cancel}
\usepackage[draft]{hyperref}
\usepackage{comment,enumitem}
\usepackage{mdframed}
\usepackage{bbm,bm}

\newcommand{\field}[1]{\mathbb{#1}}
\newcommand{\N}{\field{N}}
\newcommand{\Z}{\field{Z}}

\newcommand{\C}{\field{C}}

\renewcommand{\H}{\mathbb{H}}

\newcommand{\SL}{\operatorname{SL}}

\newcommand{\re}{\operatorname{Re}}

\newcommand{\bea}{\begin{eqnarray}}
\newcommand{\eea}{\end{eqnarray}}
\newcommand{\be}{\begin {equation}}
\newcommand{\ee}{\end{equation}}

\newcommand{\z}{\mathfrak{z}}
\newcommand{\x}{\mathbbm{x}}
\newcommand{\y}{\mathbbm{y}}

\numberwithin{equation}{section}
\newtheorem{theorem}{Theorem}
\numberwithin{theorem}{section}
\newtheorem{lemma}[theorem]{Lemma}
\newtheorem{proposition}[theorem]{Proposition}
\newtheorem{corollary}[theorem]{Corollary}

\theoremstyle{remark}

\newtheorem*{remark}{Remark}

\theoremstyle{definition}
\newtheorem*{definition}{Definition}

\begin{document}
\title[The degenerate parts of spaces of meromorphic cusp forms]{The degenerate parts of spaces of meromorphic cusp forms under a regularized inner product}
\author{Kathrin Bringmann}
\address{Department of Mathematics and Computer Science\\Division of Mathematics\\University of
Cologne\\ Weyertal 86-90 \\ 50931 Cologne \\Germany}
\email{kbringma@math.uni-koeln.de}
\author{Ben Kane}
\address{Department of Mathematics\\ University of Hong Kong\\ Pokfulam, Hong Kong}
\email{bkane@hku.hk}
\date{\today}
\thanks{The research of the first author is supported by the Alfried Krupp Prize for Young University Teachers of the Krupp foundation. The research of the second author was supported by grants from the Research Grants Council of the Hong Kong SAR, China (project numbers HKU 17316416, 17301317, and 17303618).}
\maketitle

\section{Introduction and statement of results}\label{sec:intro}
In this paper, we investigate inner products for modular forms. Petersson defined an inner product $\langle\cdot,\cdot\rangle$ between two holomorphic modular forms $f$ and $g$ that converges whenever the product $fg$ vanishes at every cusp. 
The Petersson inner product, when restricted to the space $S_{2k}$ of cusp forms of weight $2k\in2\N$, is \emph{positive definite}, i.e., for $f\in S_{2k}$,
\[
\|f\|^2:=\left<f,f\right>\geq 0\quad\text{ and }\quad \|f\|^2=0\Leftrightarrow f=0.
\]
For an inner product $[\cdot,\cdot]$, we say that $f$ and $g$ are \begin{it}orthogonal\end{it} if $[f,g]=0$ and one calls $f\neq 0$ \begin{it}isotropic\end{it} if it is orthogonal to itself. Petersson's study in \cite{PeMet} revealed that $S_{2k}$ has no isotropic elements and moreover contains an orthonormal basis with respect to the inner product. 

There are many so-called regularizations of the Petersson inner product, extending it to a bigger space $M_{2k}$ of weight $2k$ modular forms. Zagier \cite{ZagierNotRapid} included Eisenstein series $E_{2k}$ and proved that for $k$ even there are isotropic elements that lie in the space of weight $2k$ holomorphic modular forms and moreover elements $f\in M_{2k}$ with $\left\|f\right\|^2 < 0$. Petersson \cite{Pe2} defined a regularization via  Cauchy principal integrals and this idea was independently rediscovered and extended by Harvey and Moore \cite{HM} and Borcherds \cite{Borcherds} to a regularization for \emph{weakly holomorphic modular forms of weight $2k$}, i.e., those modular forms which are holomorphic on the upper half-plane but which may grow exponentially towards $i\infty$. A regularized inner product for the space $\mathbb{S}_{2k}$ of \begin{it}meromorphic cusp forms\end{it}, i.e., those meromorphic modular forms of weight $2k$ which vanish like cusp forms towards $i\infty$, was defined in \cite{BKvP}.

In this paper we consider an inner product on natural subspaces that are related the elliptic expansions of a meromorphic modular form around points $\z\in\H$. Define 
\[
\mathbb{S}_{2k}^{\z}:=\left\{ f\in \mathbb{S}_{2k}: \text{ the only possible pole of $f$ in $\SL_2(\Z)\backslash \H$ lies at $\z$}\right\}.
\]
Writing the elliptic expansion of $f\in \mathbb{S}_{2k}^{\z}$ around $\z\in\H$ as (see e.g. \cite[(7)]{Pe2})
\begin{equation}\label{eqn:fellexp}
f(z)=(z-\overline{\z})^{-2k}\sum_{n\gg -\infty}c_{f,\z}(n)X_{\z}^n(z),
\end{equation}
with $X_{\z}(z):=\frac{z-\z}{z-\overline{\z}}$, a short calculation shows that precisely the terms $1-2k\leq n\leq -1$ contribute to $\operatorname{Res}_{z=\z} f(z)$. Interpreting the residue as an integral via the Residue Theorem, there is also a well-defined residue at cusps, and the constant term in the Fourier expansion of $f$ is the only term that contributes to the residue of $f$ at $i\infty$. 
We hence 
\rm
draw a parallel between the terms $1-2k\leq n\leq -1$ in the elliptic expansion and the constant term in the Fourier expansion. Hence the projection of weakly holomorphic modular forms to $\C E_{2k}$ is completely determined by the residue. Paralelling this for meromorphic cusp forms, set
\begin{align*}
\mathbb{E}_{2k}^{\z}&:=\left\{f \in \mathbb{S}_{2k}^{\z}: c_{f,\z}(n)=0\ \forall n\leq -2k\text{ and }\left<f,g\right>=0\ \forall g\in S_{2k}\right\},\\
\mathbb{D}_{2k}^{\z}&:=\left\{f \in \mathbb{S}_{2k}^{\z}: {c_{f,\z}}(n)=0\ \forall 1-2k\leq n\leq -1\text{ and }\left<f,g\right>=0\ \forall g\in S_{2k}\right\}.
\end{align*}
For any element in $\mathbb{S}_{2k}^{\z}$ we can subtract an element of $\mathbb{D}_{2k}^{\z}$ plus an element of $\mathbb{E}_{2k}^{\z}$ to cancel its principal part in $\z$, giving a cusp form. This yields the decomposition 
\begin{equation}\label{eqn:S2ksplitting}
\mathbb{S}_{2k}^{\z}=S_{2k} \perp \left(\mathbb{E}_{2k}^{\z}\oplus \mathbb{D}_{2k}^{\z}\right).
\end{equation}
Here and throughout we write $\perp$ for an orthogonal decomposition and $\oplus$ for direct sums. We let $\mathbb{D}_{2k}$ denote the space spanned by all $\mathbb{D}_{2k}^{\z}$ with $\z\in\H$ and $\mathbb{E}_{2k}$ be the space spanned by all $\mathbb{E}_{2k}^{\z}$ with $\z\in\H$. Similarly, throughout the paper we define subspaces with singularities only possibly occurring at $\z\in\H$ and omit $\z$ in the notation for the space spanned by the union of all such subspaces with $\z\in\H$.

So-called polar harmonic cusp forms appeared in \cite{BKvP} when evaluating the inner product between certain meromorphic cusp forms arising from positive-definite quadratic forms. These polar harmonic Maass forms are pre-images of meromorphic cusp forms under $\xi_{\kappa}:=2iy^{\kappa} \overline{\frac{\partial}{\partial\overline{z}}}$ (with $z=x+iy\in \H$, $\kappa=2-2k$).

 By Lemma \ref{lem:PPsiRel} (2), (3), if $\xi_{2-2k}(F)\in S_{2k}$, then $D^{2k-1}(F)\in \mathbb{D}_{2k}$, with $D:=\frac{1}{2\pi i }\frac{\partial}{\partial z}$, and $D^{2k-1}$ is surjective onto $\mathbb{D}_{2k}$. The following theorem therefore relates the regularized inner product on $\mathbb{D}_{2k}$ with the classical inner product on $S_{2k}$.
\begin{theorem}\label{thm:xiDinner}
If $F$ and $G$ are polar harmonic cusp forms of weight $2-2k \, (k\in \mathbb N_{\geq 2})$ for which $\xi_{2-2k}(F),\xi_{2-2k}(G)\in \mathbb{D}_{2k}\perp S_{2k}$, then
\[
\left<\xi_{2-2k}(F),\xi_{2-2k}(G)\right>=-\frac{(4\pi)^{4k-2}}{(2k-2)!^2} \left<D^{2k-1}(G),D^{2k-1}(F)\right>.
\]
\end{theorem}
Theorem \ref{thm:xiDinner} leads to a number of corollaries that help to better understand the inner product on $\mathbb{S}_{2k}$. 
We first use it to understand the \begin{it}degenerate part\end{it} of $\mathbb{D}_{2k}$, i.e., the subspace of $\mathbb{D}_{2k}$ that is orthogonal to all of $\mathbb{D}_{2k}$; 
see Corollary \ref{cor:Dperp} below for a more precise version. 
\begin{corollary}\label{cor:DperpIntro}
 The degenerate part of $\mathbb{D}_{2k}$ is $D^{2k-1}(\mathbb{S}_{2-2k})$. 
\end{corollary} 
 Corollary \ref{cor:DperpIntro} reduces the understanding of the inner product on $\mathbb{D}_{2k}$ to its behaviour on $\mathbb{D}_{2k}$  $\pmod{D^{2k-1}(\mathbb{S}_{2-2k})}$. We compute the inner product on this quotient space in the following corollary; a more formal version can be found in Corollary \ref{cor:DperpFD} below.
\begin{corollary}\label{cor:DperpFDIntro}
After factoring out by $D^{2k-1}(\mathbb{S}_{2-2k})$, the inner product on $\mathbb{D}_{2k}$ coincides with the inner product on $S_{2k}$. In particular, it is positive-definite.
\end{corollary}
Corollary \ref{cor:DperpFDIntro} completely determines the inner product on $\mathbb{D}_{2k}$, so by \eqref{eqn:S2ksplitting} (extended to all $\z$) it remains to determine the relationship of inner products between elements of $\mathbb{E}_{2k}$ and $\mathbb{S}_{2k}$. As a step in this direction, for fixed $\z\in\H$ we investigate the inner product between $\mathbb{S}_{2k}^{\z}$ and other subspaces; 
see Corollary \ref{cor:HeckeExist}.
\begin{corollary}\label{cor:HeckeExistIntro}
For every $\z\in\H$, the space $\mathbb{S}_{2k}^{\z}$ factored out by its degenerate part is finite-dimensional. The space $\mathbb{D}_{2k}$ quotiented out by its subspace orthogonal to $\mathbb{S}_{2k}^{\z}$ is also finite-dimensional.
\end{corollary}
It is natural to ask if Theorem \ref{thm:xiDinner} holds more generally if the image of the polar harmonic cusp form under $\xi_{2-2k}$ is an arbitrary meromorphic cusp form. It turns out that there is a subspace $\mathscr{H}_{2-2k}^{\mathbb{E}}$ of polar harmonic Maass forms for which $\xi_{2-2k}(F),D^{2k-1}(F)\in\mathbb{E}_{2k}$; 
see \eqref{eqn:Hcuspdef} below for the definition of the subspace and Lemma \ref{lem:DXicusp} (3) for its properties. Moreover, both maps are surjective from $\mathscr{H}_{2-2k}^{\mathbb{E}}$ to $\mathbb{E}_{2k}$.  Hence, since $\mathbb{E}_{2k}$ is orthogonal to $S_{2k}$, if Theorem \ref{thm:xiDinner} would extend to $\mathbb{S}_{2k}$, then $\mathbb{E}_{2k}$ would also be orthogonal to $\mathbb{D}_{2k}$; see Lemma \ref{lem:NoExtend} for further details.
 We next see, however, that this is not the case; see Proposition \ref{prop:Eperp} for a more formal version.
\begin{proposition}\label{prop:EperpIntro}
 Theorem \ref{thm:xiDinner} does not extend to $\mathbb{S}_{2k}$. In particular, the subspace of $\mathbb{S}_{2k}$ orthogonal to all of $\mathbb{E}_{2k}$ is precisely $S_{2k}$.
\end{proposition}

The paper is organized as follows. In Section \ref{sec:prelim}, we introduce polar harmonic cusp forms more formally and recall known results about elliptic expansions and Poincar\'e series that span the spaces of polar harmonic cusp forms and meromorphic cusp forms. In Section \ref{sec:innerPoincare}, we evaluate the inner product between different Poincar\'e series. In Section \ref{sec:xiDinner}, we prove Theorem \ref{thm:xiDinner}. The corollaries of Theorem \ref{thm:xiDinner} about degenerate subspaces (i.e., Corollaries \ref{cor:DperpIntro}, \ref{cor:DperpFDIntro}, and \ref{cor:HeckeExistIntro}) and Proposition \ref{prop:EperpIntro} about the optimality of the conditions in Theorem \ref{thm:xiDinner} are proven in Section \ref{sec:Dperp}.


\section{Preliminaries}\label{sec:prelim}
\subsection{Polar harmonic Maass forms}
 We begin by defining polar harmonic Maass forms. For $M=\left(\begin{smallmatrix} a&b\\c&d\end{smallmatrix}\right)\in\operatorname{SL}_2(\mathbb{Z})$, $\kappa\in\Z$, and $F:\mathbb{H}\to\mathbb{C}$, the \emph{slash-operator} is
 \noindent
\[
F|_{2\kappa}M(z):=(cz+d)^{-2\kappa}F\left(\frac{az+b}{cz+d}\right).
\]
\begin{definition}
For $\kappa\in\mathbb{Z}$, a {\it polar harmonic Maass form of weight $2\kappa$} is a function $F:\mathbb{H}\to\mathbb{C}$, which is real-analytic outside a discrete set of $\mathbb C$ and which satisfies the following conditions:
\begin{enumerate}[leftmargin=*,label={\rm(\arabic*)}]
\item For every $M\in\operatorname{SL}_2(\mathbb{Z})$, we have $F|_{2\kappa}M=F$.
\item We have $\Delta_{2\kappa}(F)=0$, with the \textit{weight $2\kappa$ hyperbolic Laplace operator}
$$
\Delta_{2\kappa}:=-y^2\left(\frac{\partial^2}{\partial x^2}+\frac{\partial^2}{\partial y^2}\right)+2\kappa iy\left(\frac{\partial}{\partial x}+i\frac{\partial}{\partial y}\right).
$$
\item For every $\mathfrak{z}\in\mathbb{H}$, there 
exists a minimal $n_0=n_0(\z)\in\mathbb{N}_0$ such that $(z-\mathfrak{z})^{n_0}F(z)$ is bounded in some neighborhood of $\mathfrak{z}$. 
\item The function $F$ grows at most linear exponentially at the cusps.
\end{enumerate}
Polar harmonic Maass forms without singularities at $i\infty$ are called \begin{it}polar harmonic cusp forms\end{it} and the space of such forms is denoted by $\mathscr{H}_{2\kappa}$. Specifically, a polar harmonic Maass form $F$ is a polar harmonic cusp form if it is  bounded towards $i\infty$ (resp. vanishes towards $i\infty$) if $\kappa\leq 0$ (resp. $\kappa\geq 1$). 
\end{definition}
Differential operators acting on polar harmonic Maass forms play a significant role in this paper. For $\kappa\leq 0$, the operator $D^{1-2\kappa}$, which occurs in Theorem \ref{thm:xiDinner} with $\kappa=1-k$, maps weight $2\kappa$ polar harmonic Maass forms to weight $2-2\kappa$ meromorphic modular forms. The operator $\xi_{2\kappa}$ 
which also appears before Theorem \ref{thm:xiDinner},
also maps weight $2\kappa$ polar harmonic Maass forms to weight $2-2\kappa$ meromorphic modular forms. Note that $\xi_{2\kappa}(\mathscr{H}_{2\kappa})=\mathbb{S}_{2-2\kappa}$ by \cite[Theorem 1.1 (1), (4), Proposition 4.1]{BJK}. The subspace of $\mathscr{H}_{2\kappa}$ consisting of those $F$ for which $\xi_{2\kappa}(F)$ is a cusp form is denoted by $\mathscr{H}_{2\kappa}^{\operatorname{cusp}}$. 
The hyperbolic Laplace operator is related to $\xi_{2\kappa}$, $R_{2\kappa-2}$, and $L_{2\kappa}$ via
\[
\Delta_{2\kappa}=-\xi_{2-2\kappa}\circ \xi_{2\kappa}=-R_{2\kappa-2}\circ L_{2\kappa},
\]
where the \emph{raising} and \emph{lowering operators} are defined as 
\begin{equation*}
R_{2\kappa}:=2i\frac{\partial}{\partial z} +\frac{2\kappa}{y},\qquad L_{2\kappa}:=-2iy^2 \frac{\partial}{\partial \overline{z}}. 
\end{equation*}
\subsection{The Flipping operator}
For a function $F$ transforming of weight $2-2k\in -2\N_0$ define the {\it flipping operator} 
\begin{equation*}
\mathfrak{F}_{2-2k}(F):=-\frac{y^{2k-2}}{(2k-2)!}\overline{R_{2-2k}^{2k-2}(F)},
\end{equation*}
where iterated raising is defined by
\[
R_{2-2k}^n:=R_{2n-2k}\circ \cdots\circ R_{4-2k}\circ R_{2-2k}
\]
The following lemma is given in \cite[Proposition 5.15]{Book} (stated there for forms that only have singularities at $i\infty$, but the proof does not use this property).
\begin{lemma}\label{lem:flip}
\noindent

\noindent
\begin{enumerate}[leftmargin=*,label={\rm(\arabic*)}]
\item 

The operator $\mathfrak{F}_{2-2k}$ is an involution, i.e., $\mathfrak{F}_{2-2k}\circ \mathfrak{F}_{2-2k}$ is the identity. 
\item
If $\Delta_{2-2k}(F)=0$, then the $\mathfrak{F}_{2-2k}$ satisfies
\[
\xi_{2-2k}(\mathfrak{F}_{2-2k}(F))=\frac{(4\pi)^{2k-1}}{(2k-2)!} D^{2k-1}(F).
\]
\item
If $\Delta_{2-2k}(F)=0$, then we have
\[
D^{2k-1}(\mathfrak{F}_{2-2k}(F))=\frac{(2k-2)!}{(4\pi)^{2k-1}}\xi_{2-2k}(F).
\]
\end{enumerate}
\end{lemma}

\subsection{Elliptic expansions}\label{sec:elliptic}
We next consider elliptic expansions of polar harmonic Maass forms. For this, we define for $0\leq w <1$ and $a\in \N$ and $b\in \Z$
\begin{equation}\label{eqn:beta0def}
\beta_0\left(w;a,b\right):=\beta\left(w;a,b\right)-\mathcal{C}_{a,b}
\end{equation}
where $\beta({w};a,b):=\int_0^{w} t^{a-1} (1-t)^{b-1} dt$ is the \begin{it}incomplete beta function\end{it} and
\begin{equation*}
\mathcal{C}_{a,b}:=\sum_{\substack{0\leq j\leq a-1 \\ j\neq -b}} \binom{a-1}{j}\frac{(-1)^{j}}{j+b}.
\end{equation*}
Suppose that $k\in \N$ and $\mathfrak{z}\in\mathbb{H}$. 
\noindent

\noindent
\begin{enumerate}[leftmargin=*,label={\rm(\arabic*)}]
\item
Suppose that $F$ satisfies $\Delta_{2\kappa} (F) = 0$ and for some $n_0\in \N$ the function $r_{\mathfrak{z}}^{n_0}(z)F(z)$ is bounded in some neighborhood $\mathcal{N}$ around $\mathfrak{z}$, where $r_{\z}(z):=|X_{\z}(z)|$. Then there exist $c_{F,\mathfrak{z}}^{\pm}(n) \in\mathbb{C}$ such that for $z\in\mathcal{N}$ we have
\begin{multline}
\ \ F(z)=\left(z-\overline{\mathfrak{z}}\right)^{2k-2}\sum_{n\geq -n_0} c_{F,\mathfrak{z}}^+(n) X_{\mathfrak{z}}^n(z)\\ \label{eqn:elliptic}
+ \left(z-\overline{\mathfrak{z}}\right)^{2k-2}\sum_{n\leq n_0}{c_{F,\mathfrak{z}}^-}(n) \beta_0\left(1-r_{\mathfrak{z}}^2(z);2k-1,-n\right)X_{\mathfrak{z}}^n(z).
\end{multline}
\item
If $F\in\mathscr{H}_{2-2k}^{\operatorname{cusp}}$, then the second sum in \eqref{eqn:elliptic} only runs over $n<0$.
\end{enumerate}
\begin{remark}
The expansion \eqref{eqn:elliptic} was written slightly differently in \cite[Proposition 2.2]{BJK}. One obtains \eqref{eqn:elliptic}
 by plugging in \eqref{eqn:beta0def} to replace the incomplete 
$\beta$-functions appearing in \cite[Proposition 2.2]{BJK} with $\beta_0$ for $0\leq n\leq 2k-2$. 
The reason for this change of notation is that the coefficients $c_{F,\z}^+(n)$ from the expansion in \eqref{eqn:elliptic} naturally occur in our computation of the inner product in Lemma \ref{lem:innermero} below.
\end{remark}
\ \ \ \ We define the \begin{it}meromorphic part of the elliptic expansion\end{it} around $\mathfrak{z}$ by
\begin{equation}\label{eqn:elliptic+}
F_{\mathfrak{z}}^+ (z):=\left(z-\overline{\mathfrak{z}}\right)^{2k-2}\sum_{n\geq -n_0} {c_{F,\mathfrak{z}}^+}(n)X_{\mathfrak{z}}^n(z)
\end{equation}
and its \begin{it}non-meromorphic part\end{it} by 
\begin{equation}\label{eqn:elliptic-}
F_{\mathfrak{z}}^- (z) :=\left(z-\overline{\mathfrak{z}}\right)^{2k-2}
\sum_{n\leq n_0} c_{F,\mathfrak{z}}^-(n)\beta_0\left(1-r_{\mathfrak{z}}^2(z);2k-1,-n\right) X_{\mathfrak{z}}^n(z).
\end{equation}
The terms in \eqref{eqn:elliptic} which grow as $z\to\z$ are called the \begin{it}principal part of $F$ at $\z$.\end{it} Specifically, these are the terms in \eqref{eqn:elliptic+} with $n<0$ and those terms in \eqref{eqn:elliptic-} with $n\geq 0$ 
(see \cite[Lemma 5.4]{BKRohrlich}).
We furthermore define the \begin{it}polynomial part\end{it} of $F$ around $\z$ 
\begin{equation*}
p_{F,\z}(z):=\left(z-\overline{\mathfrak{z}}\right)^{2k-2}\sum_{n=0}^{2k-2}{c_{F,\mathfrak{z}}^+}(n)X_{\mathfrak{z}}^n(z).
\end{equation*}
Note that $p_{F,\z}$ is the only contribution in \eqref{eqn:elliptic} that is a polynomial in $z$. We
 show in Lemma \ref{lem:innermero} below that if $f=\xi_{2-2k}(F)\in \mathbb{E}_{2k}\oplus \mathbb{D}_{2k} $, then the polynomial part $p_{F,\z}$ naturally appears when computing the inner product between $f$ and elements of $\mathbb{E}_{2k}^{\z}$. 

Hence for $f\in \mathbb{E}_{2k}\oplus \mathbb{D}_{2k}$ and $g\in \mathbb{E}_{2k}^{\z}$, one ``only'' needs to determine $p_{F,\z}$ for some $F\in\mathscr{H}_{2-2k}$ with $\xi_{2-2k}(F)=f$ to compute the inner product $\left<f,g\right>$. However, this inner product is difficult to compute if one only works with meromorphic modular forms of weight $2k$; starting with $f$ or $D^{2k-1}(F)$, the polynomial part is somewhat mysterious because it is annihilated by $\xi_{2-2k}$ and $D^{2k-1}$. Hence one cannot immediately compute the polynomial part simply by looking at elliptic expansions of weight $2k$ meromorphic cusp forms.

\subsection{Poincar\'e series}
For $\z\in\H$, $k\in \N_{\geq 2}$, and $n\in\Z$, we define the (meromorphic) \textit{elliptic Poincar\'e series} (see for example \cite[(22) and Satz 7]{Pe2} 
\begin{equation*}
\Psi_{2k,m}^{\z}:=\sum_{M\in\operatorname{SL}_2(\mathbb{Z})}\psi_{2k,m}^{\z}|_{2k}M,
\end{equation*}
with
\[
\psi_{2k,m}^{\z}(z):=(z-\overline{\z})^{-2k}X_{\z}^m(z).
\]
The following lemma is due to Petersson (see \cite[Satz 7]{PeEinheit} and \cite[Satz 7]{Pe2}).
\begin{lemma}\label{lem:genS2k} 
\noindent

\noindent
\begin{enumerate}[leftmargin=*,label={\rm(\arabic*)}]
\item For any $\z \in\H$, the space $S_{2k}$ is spanned by $\{\Psi_{2k,m}^{\z}:m\in\mathbb{N}_0\}$.
\item For any $\z\in\H$, the set $\{\Psi_{2k,m}^{\z}: -2k<m<0\}$ is a basis for $\mathbb{E}_{2k}^{\z}$ and  $\mathbb{E}_{2k}$ is spanned by $\{\Psi_{2k,m}^{\z}:m\in\Z, 1-2k\leq m\leq -1,\z\in \H\}$.
\item  For any $\z\in\H$, a basis for $\mathbb{D}_{2k}^{\z}$ is given by $\{\Psi_{2k,m}^{\z}: m\leq -2k\}$ and $\mathbb{D}_{2k}$ is spanned by $\{\Psi_{2k,m}^{\z}:m\in\Z, m\leq -2k,\z\in \H\}$.
\end{enumerate}
\end{lemma}
\begin{proof}
Part (1)  is part of \cite[Satz 7]{PeEinheit}. 
For (2), (3), the principal parts of $\Psi_{2k,m}^{\z}$ were determined in \cite[Satz 7]{Pe2} to be constant multiples of $\psi_{2k,m}^{\z}$. Restricting $m$ as in (2) (resp. (3)) matches the principal part conditions in the definition of $\mathbb{E}_{2k}$ (resp. $\mathbb{D}_{2k}$). 
Moreover, for fixed $\z$ the forms $\Psi_{2k,m}^{\z}$ with $m<0$ are clearly linearly independent.
The orthogonality of $\Psi_{2k,m}^{\z}$ to cusp forms for $m<0$ was shown in \cite[Satz 8]{Pe2}.\qedhere
\end{proof}
Next consider \textit{harmonic elliptic Poincar\'e series} given by (\cite[(4.5), Theorem 4.3]{BJK})
\begin{equation*}
\mathbb{P}_{2-2k,m}^{\z}:=\sum_{M\in\operatorname{SL}_2(\mathbb{Z})}\varphi_{2-2k,m}^{\z}|_{2-2k}M,
\end{equation*}
where
\begin{equation*}
\varphi_{2-2k,m}^{\z}(z):=(z-\overline{\z})^{2k-2}\beta\left(1-r^2_\z(z); 2k-1, -m\right)X^m_\z(z).
\end{equation*}
For $k\in\N_{\geq 2}$, the Poincar\'e series $\mathbb{P}_{2-2k,m}^{\z}$ and $\Psi_{2k,m}^{\z}$ are related via the differential operators $\xi_{2-2k}$ and $D^{2k-1}$.
\begin{lemma}\label{lem:PPsiRel}
Let $k\in\N_{\geq 2}$.

\noindent
\begin{enumerate}[leftmargin=*,label={\rm(\arabic*)}]
\item
 We have 
\[
\xi_{2-2k}\left(\mathbb{P}_{2-2k,m}^{\z}\right)=(4\y)^{2k-1}\Psi_{2k, -m-1}^{\z}.
\]
\item
We have  
\[
D^{2k-1}\left(\mathbb{P}_{2-2k,m}^{\z}\right)=-(2k-2)!\left(\frac{\y}{\pi}\right)^{2k-1}\Psi_{2k, m+1-2k}^{\z}.
\]
\item

If  $F\in \mathscr{H}_{2-2k}$ satisfies $\xi_{2-2k}(F)\in S_{2k}$, then $D^{2k-1}(F)\in \mathbb{D}_{2k}$.
\end{enumerate}
\end{lemma}
\begin{proof}
 (1) and (2) follow by \cite[Theorem 4.3]{BJK}. \\
\noindent
(3) By \cite[Theorem 4.3]{BJK}, the space $\mathscr{H}_{2-2k}$ is spanned by the Poincar\'e series $\mathbb{P}_{2-2k,m}^{\z}$. Hence for $F\in \mathscr{H}_{2-2k}$, there exist $c_{m,\z}\in\C$ (only finitely many non-zero) for which
\begin{equation}\label{eqn:Farbitrary}
F= \sum_{m\in\Z} \sum_{\z\in\H} c_{m,\z} \mathbb{P}_{2-2k,m}^{\z}. 
\end{equation}
Using part (2), we have
\begin{equation}\label{eqn:PPsiRel2}
D^{2k-1}(F)=-(2k-2)! \sum_{\z\in\H}\left(\frac{\y}{\pi}\right)^{2k-1}\sum_{m\in\Z} c_{m,\z} \Psi_{2k,m+1-2k}^{\z}. 
\end{equation}
 Combining part (1) with Lemma \ref{lem:genS2k} (1) and noting that the intersection of $\mathbb{E}_{2k}\oplus \mathbb{D}_{2k}$ with $S_{2k}$ is trivial by definition, \eqref{eqn:Farbitrary} implies that $\xi_{2-2k}(F)\in S_{2k}$ if and only if 
\[
\sum_{m\geq 0}\sum_{\z\in\H}c_{m,\z}\mathbb{P}_{2-2k,m}^{\z}=0.
\]
Under this restriction, \eqref{eqn:PPsiRel2} and Lemma \ref{lem:genS2k} (3) imply that $D^{2k-1}(F)\in \mathbb{D}_{2k}$. \qedhere
\end{proof}

Lemma \ref{lem:PPsiRel} (1) yields a relationship between the elliptic coefficients of $\mathbb{P}_{2-2k,n}^{\z}$ and $\Psi_{2k,m}^{\z}$ for certain $n,m$.
Define $\delta_{\z_0,\z}:=1$ if $\z_0$ is equivalent under the action of $\SL_2(\Z)$ to $\z$ and $\delta_{\z_0,\z}:=0$ otherwise and let $\omega_{\mathfrak{z}}:=\#\Gamma_{\mathfrak{z}}$ be the size of the stabilizer group $\Gamma_{\z}$ of $\mathfrak{z}$ in $\operatorname{PSL}_2(\mathbb{Z})$. The elliptic expansion of $\Psi_{2k,m}^{\z}$ is given by  (see \cite[(2.26)]{BKvP}) 

\begin{equation*}
\Psi_{2k,m}^{\z}(z)=2\omega_{\z}\left(z-\overline{\varrho}\right)^{-2k}\left(\delta_{\z,\varrho}X_{\varrho}^{m}(z)+\sum_{n\geq 0}c_{m,\varrho}^{\z}(n) X_{\varrho}^{n}(z)\right).
\end{equation*}
 For $\z_0\in\H$, as in \eqref{eqn:elliptic} we write (see \cite[Theorem 4.3]{BJK} for the principal part)
\begin{align}\nonumber
	&\mathbb{P}_{2-2k,m+k-1}^{\z_0}(z)
	= 2\omega_{\z_0}\delta_{\z_0,\z} (z-\overline{\z})^{2k-2} \bigg(\delta_{m<1-k} \mathcal{C}_{2k-1,1-k-m}\\\nonumber
 &\hspace{30pt}+\delta_{m\geq 1-k}\beta_0\!\left(1-r_{\z}^2(z);2k-1,1-k-m\right)\bigg)\!X_{\z}^{m+k-1}(z)
 +2\omega_{\z_0}(z-\overline{\z})^{2k-2}\\\label{eqn:expandPgen}
	&\times\left(\sum_{n\geq 0} c_{m,\z}^{\z_0,+}(n) X_{\z}^n(z) + \sum_{n\leq -1} c_{m,\z}^{\z_0,-}(n)\beta_0\!\left(1-r_{\z}^2(z);2k-1,-n\right)X_{\z}^n(z)\!\right).
\end{align}
For $F=\mathbb{P}_{2-2k,k-1\pm m}^{\z_0}$ and $\z\in\H$, we have $n_0(\z)=0$ if $\z$ is not equivalent under $\SL_2(\Z)$ to $\z_0$ and by \cite[Lemma 5.4]{BKRohrlich} and \cite[Theorem 4.3]{BJK} we have $n_0(\z_0)=\left|k-1\pm m\right|+\delta_{k-1=\mp m}$. Set 
\begin{multline}\label{eqn:bdef}
b_{m,\z}^{\z_0}(n)\\:=\begin{cases}
-\frac{(-n+2k-2)!}{(-n-1)!}
2\omega_{\z_0}\delta_{\z_0,\z}\delta_{n=m+k-1}\mathcal{C}_{2k-1,1-k-m}&\text{if }n<0,\\
-2(2k-2)!\omega_{\z_0}\delta_{\z_0,\z}\delta_{n=k-1+m} &\text{if }0\leq n\leq \min(n_0,2k-2),\\
2\omega_{\z_0}
\frac{n!}{(n+1-2k)!}c_{m,\z}^{\z_0,+}(n)&\text{if }n\geq 2k-1.
\end{cases}
\end{multline}
This constant appears for a more general $F\in \mathscr{H}_{2-2k}$ in \cite[Proposition 2.3]{BJK}. We have the following relationship between the elliptic coefficients.
\begin{lemma}\label{lem:xiDelliptic}
\noindent

\noindent
\begin{enumerate}[leftmargin=*,label={\rm(\arabic*)}]
\item
We have
\begin{align*}
\nonumber &\xi_{2-2k}\left(\mathbb{P}_{2-2k,k-1-m}^{\z_0}(z)\right)\\\nonumber&\hspace{50pt}=\frac{(4\y)^{2k-1}2\omega_{\z_0}}{(z-\overline{\z})^{2k}}\sum_{n\leq n_0} \left(\overline{c_{-m,\z}^{\z_0,-}(n)}+\delta_{\z,\z_0}\delta_{-m>k-1}\delta_{n=k-1-m}\right)X_{\z}^{-n-1}(z),\\
\label{eqn:DEllipticExp} &D^{2k-1}\left(\mathbb{P}_{2-2k,k-1+m}^{\z_0}(z)\right)=\left(\frac{\y}{\pi}\right)^{2k-1}(z-\overline{\z})^{-2k}\sum_{n\geq -n_0} b_{m,\z}^{\z_0}(n) X_{\z}^{n+1-2k}(z).
\end{align*}

\item 
For $n\leq n_0$, we have 
\[
\overline{c_{-m,\z}^{\z_0,-}(n)}+\delta_{\z,\z_0}\delta_{-m>k-1}\delta_{n=k-1-m}= \left(\frac{\y_0}{\y}\right)^{2k-1}\left(c_{m-k,\z}^{\z_0}(-n-1)+\delta_{\z,\z_0}\delta_{n=k-1-m}\right). 
\]
\item  For $n\geq -n_0$ we have, where here and throughout for $\z_j\in\H$ we write $\z_j=\x_j+i\y_j$,
\[
 b_{m,\z}^{\z_0}(n)= -(2k-2)!\left(\frac{\y_0}{\y}\right)^{2k-1}2\omega_{\z_0}\left(c_{m-k,\z}^{\z_0}(n+1-2k)+\delta_{\z,\z_0}\delta_{n=k-1+m}\right).
\]

\end{enumerate}
\end{lemma}
 \begin{proof}
(1) Note that for  $0\leq n\leq 2k-2$,  $(z-\overline{\z})^{2k-2}X_{\z}^n(z)$
is a polynomial in $z$ of degree at most $2k-2$. Therefore it is annihilated by both $\xi_{2-2k}$ and $D^{2k-1}$, so the change in the definition of the elliptic expansion \eqref{eqn:elliptic} in comparison with \cite[Proposition 2.2]{BJK} does not change the elliptic expansions of the images under $\xi_{2-2k}$ or $D^{2k-1}$ of a polar harmonic Maass form. 
The claim then follows directly by plugging in \cite[Proposition 2.3]{BJK}. \\
\rm
\noindent
(2) The claim follows directly from part (1) and Lemma \ref{lem:PPsiRel} (1).\\
\noindent
(3) The claim follows directly from part (1) and Lemma \ref{lem:PPsiRel} (2). 
\rm
\end{proof}

\
\subsection{Spaces of polar harmonic Maass forms and differential operators}
Consider the following subspaces of $\mathscr{H}_{2-2k}^{\z}$:
\begin{equation}\label{eqn:Hcuspdef}
\mathscr{H}_{2-2k}^{\z,\operatorname{cusp}}:=\bigoplus_{m\geq k} \C \mathbb{P}_{2-2k,k-1-m}^{\z}, \qquad \mathscr{H}_{2-2k}^{\z,\mathbb{E}}:=\bigoplus_{|m|<k} \C \mathbb{P}_{2-2k,k-1+m}^{\z}.
\end{equation}
The definitions of these spaces are motivated by the following lemma. 
\begin{lemma}\label{lem:DXicusp}
Let $F\in\mathscr{H}_{2-2k}$ be given.

\begin{enumerate}[leftmargin=*,label={\rm(\arabic*)}]
\item
We have $\xi_{2-2k}(F)\in S_{2k}$ if and only if $F\in \mathscr{H}_{2-2k}^{\operatorname{cusp}}$. Moreover, $\xi_{2-2k}$ is surjective onto $S_{2k}$. 
\item
We have $D^{2k-1}(F)\in \mathbb{D}_{2k}$ if and only if $F\in \mathscr{H}_{2-2k}^{\operatorname{cusp}}\oplus \ker(D^{2k-1})$. Moreover, $D^{2k-1}$ is surjective onto $\mathbb{D}_{2k}$. 
\item 
For $F\in\mathscr{H}_{2-2k}^{\z,\mathbb{E}}$, we have $\xi_{2-2k}(F), D^{2k-1}(F)\in \mathbb{E}_{2k}^{\z}$. Moreover, both $\xi_{2-2k}$ and $D^{2k-1}$ are surjective restricted to $\mathscr{H}_{2-2k}^{\mathbb{E}}$.
\end{enumerate}
\end{lemma}
\begin{proof}
(1) We may write
\begin{equation*}
F= \sum_{m\in\Z} \sum_{\z\in\H} c_{k-1-m,\z} \mathbb{P}_{2-2k,k-1-m}^{\z}. 
\end{equation*}
 By Lemma \ref{lem:PPsiRel} (1), we have 
\[
\xi_{2-2k}(F)= \sum_{m\in\Z} \sum_{\z\in\H} \overline{c_{k-1-m,\z}} (4\y)^{2k-1} \Psi_{2k,m-k}^{\z}. 
\]
We then split $F=F_1+F_2+F_3$ with
\begin{align*}
F_1&:= \sum_{m\geq k} \sum_{\z\in\H} c_{k-1-m,\z} \mathbb{P}_{2-2k,k-1-m}^{\z}\in\mathscr{H}_{2-2k}^{\operatorname{cusp}},\quad
F_2:= \sum_{-k<m<k} \sum_{\z\in\H} c_{k-1-m,\z} \mathbb{P}_{2-2k,k-1-m}^{\z},\\
F_3&:= \sum_{m\leq -k} \sum_{\z\in\H} c_{k-1-m,\z} \mathbb{P}_{2-2k,k-1-m}^{\z}. 
\end{align*}
Lemma \ref{lem:PPsiRel} (1) and Lemma \ref{lem:genS2k} (1) then imply that $\xi_{2-2k}(F_1)\in S_{2k}$, and hence $\xi_{2-2k}(F)\in S_{2k}$ if and only if $\xi_{2-2k}(F_2+F_3)\in S_{2k}$. 
By Lemma \ref{lem:PPsiRel} (1) and Lemma \ref{lem:genS2k} (2), (3), we have $\xi_{2-2k}(F_2+F_3)\in \mathbb{E}_{2k}\oplus \mathbb{D}_{2k}$. Since $\mathbb{E}_{2k}\oplus \mathbb{D}_{2k} \cap S_{2k}=\lbrace 0 \rbrace$, we require that $\xi_{2-2k}(F_2+F_3)=0$. The kernel of $\xi_{2-2k}$ inside $\mathscr{H}_{2-2k}$ is $\mathbb{S}_{2-2k}$ and \eqref{eqn:expandPgen} implies that the principal part of $F_2+F_3$ is non-meromorphic if it is nonzero. Thus $F_2+F_3\in\mathbb{S}_{2-2k}$ if and only if $F_2+F_3=0$. 

\noindent (2) 
We again split $F=F_1+F_2+F_3$ with $F_j$ as in part (1). Since $\xi_{2-2k}(F_1)\in S_{2k}$ by part (1), Lemma \ref{lem:PPsiRel} (3) implies that $D^{2k-1}(F_1)\in \mathbb{D}_{2k}$. Thus $D^{2k-1}(F)\in \mathbb{D}_{2k}$ if and only if $D^{2k-1}(F_2+F_3)=D^{2k-1}(F-F_1)\in \mathbb{D}_{2k}$. Since the projection of $F_2+F_3$ to $\mathscr{H}_{2-2k}^{\operatorname{cusp}}$ is trivial, the first claim is equivalent to showing that $D^{2k-1}(F_2+F_3)\in \mathbb{D}_{2k}$ if and only if $F_2+F_3\in \ker(D^{2k-1})$. Lemma \ref{lem:PPsiRel} (2) implies that 
\[
D^{2k-1}(F_2+F_3)=-(2k-2)! \sum_{m<k} \sum_{\z\in\H} c_{k-1+m,\z} \left(\frac{\y}{\pi}\right)^{2k-1} \Psi_{2k,-m-k}^{\z}. 
\]
By Lemma \ref{lem:genS2k} (1), (2), $D^{2k-1}(F_2+F_3)\in \mathbb{E}_{2k}\perp S_{2k}$. 
 Since the intersection of $\mathbb{E}_{2k}\perp S_{2k}$ with $\mathbb{D}_{2k}$ is trivial by the splitting in \eqref{eqn:S2ksplitting},  $D^{2k-1}(F_2+F_3)\in \mathbb{D}_{2k}$ if and only if $F_2+F_3\in \ker(D^{2k-1})$. 

The surjectivity of the map $D^{2k-1}$ follows from Lemma \ref{lem:PPsiRel} (2) and Lemma \ref{lem:genS2k} (3) by taking a spanning set $F=\mathbb{P}_{2-2k,k-1-m}^{\z}$ with $m\geq k$ and $\z\in\H$.

\noindent
(3)
An arbitrary element of $\mathscr{H}_{2-2k}^{\z,\mathbb{E}}$ is of the form 
\begin{equation*}\label{eqn:FE2k}
F=  \sum_{0\leq m\leq 2k-2} c_{m,\z} \mathbb{P}_{2-2k,m}^{\z}. 
\end{equation*}
Lemma \ref{lem:PPsiRel} (1) and Lemma \ref{lem:genS2k} (2) imply  that $\xi_{2-2k}(F)\in \mathbb{E}_{2k}^{\z}$, while Lemma \ref{lem:genS2k} (2) and Lemma \ref{lem:PPsiRel} (2) imply that $D^{2k-1}(F)\in \mathbb{E}_{2k}^{\z}$. Lemma \ref{lem:genS2k} (2) furthermore implies that both of these maps are surjective. \qedhere
\end{proof}

\section{Inner products with Poincar\'e series}\label{sec:innerPoincare}
We recall the regularization from \cite[Section 3.2]{BKvP} for meromorphic cusp forms. For $f,g\in\mathbb{S}_{2k}$ with poles at $\mathfrak{z}_\ell (1\leq \ell\leq r)\in \SL_2(\Z)\backslash\H$, we choose a fundamental domain $\mathcal{F}^*$ such that $\mathfrak{z}_\ell\in\mathcal{F}^*$ (also denoted by $\mathfrak{z}_{\ell}$) all lie in the interior of $\Gamma_{\mathfrak{z}_\ell}
\mathcal{F}^*
$.

For an analytic function $A(\bm{s})$ in $\bm{s}=(s_1,\dots,s_r)$, denote by $\mathrm{CT}_{\bm{s}=\bm{0}}A(\bm{s})$  the constant term of the meromorphic continuation of $A(\bm{s})$ around $\bm s=\bm 0$, and define
\begin{equation}\label{eqn:OurReg}
\left<f,g\right>:= \operatorname{CT}_{\bm{s}=0}\left(\int_{\SL_2(\Z)\backslash \H} f(z) H_{\bm{s}}(z) \overline{g(z)} y^{2k}\frac{dx dy}{y^2}\right),
\end{equation}
where
\[
H_{\bm{s}} (z) = H_{s_1,\dots, s_r, \mathfrak{z}_1,\dots,\mathfrak{z}_r} (z) := \prod_{\ell=1}^r h_{s_{\ell},\mathfrak{z}_\ell} (z).
\]
Here for $\mathfrak{z}_\ell\in \mathcal{F}^*$ and $z\in\mathbb{H}$ we set $h_{s_{\ell},\mathfrak{z}_\ell}(z):=r_{\mathfrak{z}_\ell}^{2s_{\ell}}(\gamma z)$, 
with $\gamma \in \SL_2(\Z)$ such that $\gamma z\in \mathcal{F}^*$.  Note that  $r_{\mathfrak{z}_\ell}(\gamma z)\to 0$ as $z\to \gamma^{-1}\mathfrak{z}_\ell$, so the integral in \eqref{eqn:OurReg} converges for $\bm{\sigma}\gg\bm{0}$, where this notation means that for every $1\leq \ell\leq r$, $\sigma_{\ell}:=\re(s_{\ell})\gg 0$.  One can show that the regularization is independent of the choice of fundamental domain. Proceeding as in the proof of  \cite[Theorem 6.1]{BKvP} a lengthy calculation gives the following lemma.
\begin{lemma}\label{lem:innermero}
If $f\in \mathbb{E}_{2k}\oplus\mathbb{D}_{2k}$ and $F\in \mathscr{H}_{2-2k}$ satisfies $\xi_{2-2k}(F)=f$, then for all $\ell\in\Z$ 
\begin{equation*}
\left<\Psi_{2k,\ell}^{\z},f\right>=\frac{2\pi}{\y}\delta_{\ell\leq-1} c_{F,\z}^+(-\ell-1).
\end{equation*}
\end{lemma}
For $\ell\geq 0$, we recover the fact that $\mathbb{E}_{2k}\oplus \mathbb{D}_{2k}$ is orthogonal to cusp forms from Lemma \ref{lem:innermero} and Lemma \ref{lem:genS2k} (1). Lemma \ref{lem:innermero} does not yield the value of the inner product between different elements of $S_{2k}$. For this we require the Petersson coefficient formula for elliptic Poincar\'e series. Using the notation in \eqref{eqn:fellexp}, Petersson \cite[Satz 9]{Pe2} proved the following.
\begin{lemma}\label{lem:PeterssonCoeff}
If $f\in S_{2k}$, then for $n\in\N_0$ we have
\begin{equation*}
\left\langle f,\Psi_{2k,n}^{\z}\right\rangle= \frac{8\pi(2k-2)!n!}{\left(4\y\right)^{2k}(2k-1+n)!}c_{f,\z}(n).
\end{equation*}
In particular, if $\z=\z_1$ and $f=\Psi_{2k,m}^{\z_2}$ with $m\in\N_0$, then
\[
\left\langle\Psi_{2k,m}^{\z_2},\Psi_{2k,n}^{\z_1}\right\rangle =  \frac{8\pi(2k-2)!n!}{\left(4\y_1\right)^{2k}(2k-1+n)!}2\omega_{\z_2}\left(c_{m,\z_1}^{\z_2}(n)+\delta_{\z_1,\z_2}\delta_{m=n}\right).
\]
\end{lemma}

\section{Proof of Theorem \ref{thm:xiDinner}}\label{sec:xiDinner}
We are now ready to prove Theorem \ref{thm:xiDinner}. 
\begin{proof}[Proof of Theorem \ref{thm:xiDinner}]
By Lemma \ref{lem:PPsiRel} (1) and Lemma \ref{lem:genS2k}, we may assume that $F=\mathbb{P}_{2-2k,n}^{\z_1},$ $G=\mathbb{P}_{2-2k,m}^{\z_2}$ with $n,m\notin[0,2k-2]$. By Lemma \ref{lem:PPsiRel} (1), (2), changing $n\mapsto n+2k-1$ and $m\mapsto m+2k-1$, we need to show that 
\begin{equation}\label{eqn:PsiInnertoshow}
\left\langle\Psi_{2k,m}^{\z_2},\Psi_{2k,n}^{\z_1}\right\rangle=-\left\langle\Psi_{2k,-n-2k}^{\z_1},\Psi_{2k,-m-2k}^{\z_2}\right\rangle.
\end{equation}
After the change of variables, the restrictions on $n$ and $m$ become $n,m\notin [1-2k,-1]$. Noting the symmetry in \eqref{eqn:PsiInnertoshow}, we may assume without loss of generality that $n\geq 0$, and hence $\Psi_{2k,n}^{\z_1}$ is a cusp form by Lemma \ref{lem:genS2k} (1).  Petersson showed in \cite[Satz 8]{Pe2} that $\Psi_{2k,m}^{\z}$ is orthogonal to cusp forms if $m\leq-1$. Hence if $m\leq -2k$, then both sides of \eqref{eqn:PsiInnertoshow} vanish, and we may assume that $n,m\geq 0$. 
Setting $\mathcal{P}:=(4\y_2)^{1-2k}\mathbb{P}_{2-2k,2k-1+m}^{\z_2}$, we see by Lemma \ref{lem:PPsiRel} (1) that \eqref{eqn:PsiInnertoshow} is
 equivalent to
\begin{equation}\label{inner}
\left\langle\Psi_{2k,m}^{\z_2},\Psi_{2k,n}^{\z_1}\right\rangle=-\left\langle\Psi_{2k,-n-2k}^{\z_1},\xi_{2-2k}(\mathcal{P})\right\rangle.
\end{equation}
Since $m\geq 0$, Lemma \ref{lem:genS2k} (3) implies that $\Psi_{2k,-m-2k}^{\z_2}\in \mathbb{D}_{2k}$ and hence we may apply  Lemma \ref{lem:innermero} to see that \eqref{inner} is equivalent to 
\begin{equation*}\label{eqn:InnerBeforePeterssonCoeffFormula}
\left\langle\Psi_{2k,m}^{\z_2},\Psi_{2k,n}^{\z_1}\right\rangle=- \frac{2\pi }{\y_1}c_{\mathcal{P},\z_1}^+(2k-1+n)=-\frac{2\pi (4\y_2)^{1-2k}}{\y_1}2\omega_{\z_2}c_{k+m,\z_1}^{\z_2,+}(2k-1+n),
\end{equation*}
where the last equality follows by recalling from \eqref{eqn:expandPgen} that $c_{j,\z_1}^{\z_2,+}(\ell)$ is the $\ell$-th coefficient in the elliptic expansion of $\mathbb{P}_{2-2k,j+k-1}^{\z_2}(z)$ around $\z_1$ and plugging back in the definition of $\mathcal{P}$. By Lemma \ref{lem:PeterssonCoeff} the right-hand side equals
\begin{equation*}
 \frac{8\pi(2k-2)!n!}{(4\y_1)^{2k}(2k-1+n)!} \left(c_{m,\z_1}^{\z_2}(n)+\delta_{\z_1,\z_2}\delta_{m=n}\right) = -\frac{2\pi \left(4\y_2\right)^{1-2k}}{\y_1}c_{k+m,\z_1}^{\z_2,+}(2k-1+n). 
\end{equation*}
By Lemma \ref{lem:xiDelliptic} (3) with $m\mapsto m+k$ and $n\mapsto 2k-1+n$, we have
\begin{equation}\label{eqn:almostthere}
b_{m+k,\z_1}^{\z_2}(2k-1+n) =-(2k-2)!\left(\frac{\y_2}{\y_1}\right)^{2k-1}2\omega_{\z_2}\left(c_{m,\z_1}^{\z_2}(n)+\delta_{\z_1,\z_2}\delta_{m=n}\right).
\end{equation}
Since $2k-1+n\geq 2k-1$ due to $n\geq 0$, by the definition \eqref{eqn:bdef} we have 
\[
b_{m+k,\z_1}^{\z_2}(2k-1+n)= 2\omega_{\z_2}\frac{(n+2k-1)!}{n!} c_{m+k,\z_1}^{\z_2,+}(2k-1+n).
\]
Plugging this into \eqref{eqn:almostthere} yields
\begin{multline*}
2\omega_{\z_2}\frac{(n+2k-1)!}{n!} c_{m+k,\z_1}^{\z_2,+}(2k-1+n)\\=
-(2k-2)!\left(\frac{\y_2}{\y_1}\right)^{2k-1} 2\omega_{\z_2}\left(c_{m,\z_1}^{\z_2}(n)+\delta_{\z_1,\z_2}\delta_{m=n}\right).
\end{multline*}
Combining yields the claim.
\end{proof}
\section{Proof of Corollaries \ref{cor:DperpIntro}, \ref{cor:DperpFDIntro}, and \ref{cor:HeckeExistIntro}, and Proposition \ref{prop:EperpIntro}}\label{sec:Dperp}
Theorem \ref{thm:xiDinner} yields interesting corollaries about orthogonality between different spaces. More precisely, for a space $Y$, we investigate, for $X$ a subspace of $Y$
\[
X^{\perp}=X_Y^{\perp}:=\{ f\in Y:\left[f,g\right]=0\text{ for all } g\in X\}\subseteq Y.
\]
We may omit the dependence on $Y$ whenever it is clear from the context.

Many natural questions boil down to determining $X_{Y}^{\perp}$. For example, taking $X=\C f$ to be the space spanned by a fixed $f$, we see that $f$ is isotropic if and only if $f\in (\C f)^{\perp}$. Moreover, note that $(\C f)_Y^{\perp}=Y$ if and only if $f\in Y_Y^{\perp}=Y^{\perp}.$ We call $Y^{\perp}$ the \begin{it}degenerate part\end{it} of $Y$. 
Our first corollary, a formal version of Corollary \ref{cor:DperpIntro}, yields an explicit evaluation of the degenerate part of $\mathbb{D}_{2k}$.
\begin{corollary}\label{cor:Dperp}
We have 
\[
\left(\mathbb{D}_{2k}\right)_{\mathbb{D}_{2k}}^{\perp}= D^{2k-1}\!\left(\mathbb{S}_{2-2k}\right).
\]
\end{corollary}
\begin{proof}
 By Lemma \ref{lem:DXicusp} (2), for $g\in \mathbb{D}_{2k}$, we may choose $G\in \mathscr{H}_{2-2k}^{\operatorname{cusp}}$ such that $D^{2k-1}(G)=g$. Thus $D^{2k-1}(F)\in (\mathbb{D}_{2k})_{\mathbb{D}_{2k}}^{\perp}$ if and only if for every $G\in \mathscr{H}_{2-2k}^{\operatorname{cusp}}$ we have $\langle D^{2k-1}(F),D^{2k-1}(G)\rangle=0.$ By Theorem \ref{thm:xiDinner}, this is equivalent to\\ 
$\langle\xi_{2-2k}(F),\xi_{2-2k}(G)\rangle=0$ for all $G\in\mathscr{H}_{2-2k}^{\operatorname{cusp}}$. From Lemma \ref{lem:DXicusp} (1), this is equivalent to $f:=\xi_{2-2k}(F)\in S_{2k}$ being orthogonal to all of $S_{2k}$. Since the inner product is positive-definite on $S_{2k}$, $f$ is orthogonal to all of $S_{2k}$ if and only if $f=0$, which holds if and only if $F\in \mathbb{S}_{2-2k}$ (because the kernel of $\xi_{2-2k}$ is the subspace of meromorphic modular forms). Hence $D^{2k-1}(F)\in (\mathbb{D}_{2k})_{\mathbb{D}_{2k}}^{\perp}$ if and only if $F\in  \mathbb{S}_{2k}$, which is the statement of Corollary \ref{cor:Dperp}.
\end{proof}

The next corollary is a formal version of Corollary \ref{cor:DperpFDIntro}.
\begin{corollary}\label{cor:DperpFD}
We have $\dim(\mathbb{D}_{2k}/(\mathbb{D}_{2k})_{\mathbb{D}_{2k}}^{\perp})=\dim(S_k)$  and $\langle\cdot,\cdot\rangle$ is 
positive-definite on  $\mathbb{D}_{2k}/(\mathbb{D}_{2k})_{\mathbb{D}_{2k}}^{\perp}$.
\end{corollary}
\begin{proof}
By Lemma \ref{lem:DXicusp} (2), for $f\in \mathbb{D}_{2k}$ there exists $F\in \mathscr{H}_{2-2k}^{\operatorname{cusp}}$ such that $D^{2k-1}(F)=f$. Setting $\mathbbm{f}:=\xi_{2-2k}(F)\in S_{2k}$, Theorem \ref{thm:xiDinner} implies that 
\[
\left<f,f\right>=\left<f,D^{2k-1}(F)\right>=-\frac{(2k-2)!^2}{(4\pi)^{4k-2}} \left<\xi_{2-2k}(F),\mathbbm{f}\right>=\left<\mathbbm{f},\mathbbm{f}\right>.
\]
Since the inner product is positive-definite on $S_{2k}$, $\langle\mathbbm{f},\mathbbm{f}\rangle\geq 0$, with equality if and only if $\mathbbm{f}=0$. Since $\mathbbm{f}=0$ is equivalent to $f\in (\mathbb{D}_{2k})_{\mathbb{D}_{2k}}^{\perp}$ by Corollary \ref{cor:Dperp}, the inner product is positive-definite on $\mathbb{D}_{2k}/(\mathbb{D}_{2k})_{\mathbb{D}_{2k}}^{\perp}$.  If $\mathbbm{f}_1,\dots,\mathbbm{f}_d$ form an orthogonal basis of $S_{2k}$ and $F_1,\dots,F_d\in\mathscr{H}_{2-2k}^{\operatorname{cusp}}$ satisfy $\xi_{2-2k}(F_d)=\mathbbm{f}_d$, then Theorem \ref{thm:xiDinner} and Lemma \ref{lem:DXicusp} (2) imply that $f_1:=D^{2k-1}(F_1),\dots,f_d:=D^{2k-1}(F_d)$ form an orthogonal basis of $\mathbb{D}_{2k}/(\mathbb{D}_{2k})_{\mathbb{D}_{2k}}^{\perp}$. The dimensions hence coincide. 
\end{proof}

We next use Theorem \ref{thm:xiDinner} to evaluate $(\mathbb{S}_{2k}^{\z})_{Y}^{\perp}$ for certain $Y$. In order to prove Corollary \ref{cor:HeckeExistIntro}, we first require the following corollary.
\begin{corollary}\label{cor:Hecke}
\noindent

\noindent
\begin{enumerate}[leftmargin=*,label={\rm(\arabic*)}]
\item For any $\varrho\in \H$ we have
\begin{multline*}
\left(\mathbb{S}_{2k}^{\z}\right)_{\mathbb{D}_{2k}^{\varrho}}^{\perp} = \left\{ f\in D^{2k-1}\left(\mathbb{S}_{2-2k}^{\varrho}\right): f=D^{2k-1}(F)\text{ with }F\in\mathbb{S}_{2-2k}^{\varrho}, \ p_{\mathfrak{F}_{2-2k}(F),\z}(z)=0\right\}.
\end{multline*}
\item We have
\begin{multline*}
\left(\mathbb{S}_{2k}^{\z}\right)_{\mathbb{D}_{2k}}^{\perp} = \left\{ f\in D^{2k-1}\left(\mathbb{S}_{2-2k}^{\z}\right): f=D^{2k-1}(F)\text{ with }F\in\mathbb{S}_{2-2k}^{\varrho}, \ p_{\mathfrak{F}_{2-2k}(F),\z}(z)=0\right\}.
\end{multline*}
\end{enumerate}
\end{corollary}
\begin{proof}
(1) Note first that if $f\in (\mathbb{S}_{2k}^{\z})_{\mathbb{D}_{2k}^{\varrho}}^{\perp}$, then in particular $f\in (\mathbb{D}_{2k}^{\z})_{\mathbb{D}_{2k}^{\varrho}}^{\perp}$. We claim that $f\in (\mathbb{D}_{2k})_{\mathbb{D}_{2k}^{\varrho}}^{\perp}$, which by Corollary \ref{cor:Dperp} implies that $f\in D^{2k-1}(\mathbb{S}_{2-2k}^{\varrho})$. 

In order to show that $f\in (\mathbb{D}_{2k})_{\mathbb{D}_{2k}^{\varrho}}^{\perp}$, we claim that there is a  set of representatives  $f_1,\dots,f_d\in \mathbb{D}_{2k}^{\z}$ which form an orthogonal basis of $\mathbb{D}_{2k}/(\mathbb{D}_{2k})_{\mathbb{D}_{2k}}^{\perp}$, where $d=\dim_{\C}(S_{2k})$. If this is the case, then for $f\in (\mathbb{D}_{2k}^{\z})_{\mathbb{D}_{2k}^{\varrho}}^{\perp}$ we have $\langle f,f_j\rangle=0$ for all $1\leq j\leq d$. Since 
$
\langle f,f_j+g\rangle=\langle f,f_j\rangle
$
 for every $g\in \mathbb{D}_{2k}^{\perp}$, we see that this implies that $f$ is orthogonal to all of $\mathbb{D}_{2k}$ if there is indeed such an orthogonal basis $f_1,\dots,f_d$. To prove the existence of the basis elements $f_1,\dots f_d$, note that by Lemma \ref{lem:genS2k}, $S_{2k}$ is spanned by $\Psi_{2k,n}^{\z}$ with $n\in N_0$. Thus there exists an orthogonal basis $g_1,\dots,g_d$ of $S_{2k}$ and some $c_{j,n}\in \C$ (with only finitely many non-zero) $g_{j}=\sum_{n\geq0} c_{j,n} \Psi_{2k,n}^{\z}$. By Lemma \ref{lem:PPsiRel} (1), we have 
\[
g_j= \xi_{2-2k}\left((4\y)^{1-2k} \sum_{n\geq0} \overline{c_{j,n}} \mathbb{P}_{2-2k,-n-1}^{\z}\right). 
\]
We then set 
\[
f_j:=D^{2k-1}\left((4\y)^{1-2k} \sum_{n\geq0} \overline{c_{j,n}} \mathbb{P}_{2-2k,-n-1}^{\z}\right). 
\]
By Theorem \ref{thm:xiDinner}, 
\[
\left<f_j,f_{\ell}\right>=-\frac{(2k-2)!^2}{(4\pi)^{4k-2}} \left<g_{\ell},g_j\right>,
\]
and we see that $f_1,\dots,f_d$ is hence an orthogonal basis. 
We therefore conclude that $f\in  (\mathbb{D}_{2k}^{\z})_{\mathbb{D}_{2k}^{\varrho}}^{\perp}$ if and only if $f\in (\mathbb{D}_{2k})_{\mathbb{D}_{2k}^{\varrho}}^{\perp}$, which by Corollary \ref{cor:Dperp} is equivalent to $f\in D^{2k-1}(\mathbb{S}_{2-2k})$.  
We thus assume that $f\in D^{2k-1}(\mathbb{S}_{2-2k})$ and since $D^{2k-1}(\mathbb{S}_{2-2k})$ is orthogonal to cusp forms by \cite[Satz 8]{Pe2}, we see that $f=D^{2k-1}(F) \in (\mathbb{S}_{2k}^{\z})_{\mathbb{D}_{2k}^{\varrho}}^{\perp}$ if and only if $f\in (\mathbb{E}_{2k}^{\z})_{\mathbb{D}_{2k}^{\varrho}}^{\perp}$. It remains to show that $p_{\mathfrak{F}_{2-2k}(F),\z}=0$ if $f=D^{2k-1}(F) \in (\mathbb{E}_{2k}^{\z})_{\mathbb{D}_{2k}^{\varrho}}^{\perp}$. Since $f=D^{2k-1}(F)$, Lemma \ref{lem:flip} (2) implies that $f= \frac{(2k-2)!}{(4\pi)^{2k-1}} \xi_{2-2k}\left(\mathfrak{F}_{2-2k}(F)\right).$
Lemma \ref{lem:innermero} then gives that for $1-2k\leq n\leq -1$
\begin{equation}\label{eqn:innerE}
\left<\Psi_{2k,n}^{\z}, f\right>=0
\end{equation}
if and only if 
\begin{equation}\label{eqn:innerEcoeff}
c_{\mathfrak{F}_{2-2k}(F),\z}^{+}(-n-1)=0.
\end{equation}
This holds for every $1-2k\leq n\leq -1$ if and only if $p_{\mathfrak{F}_{2-2k}(F),\z}(z)=0$.

\noindent
(2) Analogously to (1), since $f\in (\mathbb{S}_{2k}^{\z})_{\mathbb{D}_{2k}}^{\perp}$, we have $f\in (\mathbb{D}_{2k}^{\z})_{\mathbb{D}_{2k}}^{\perp}$. Since there exists an orthogonal basis $f_1,\dots, f_d\in \mathbb{D}_{2k}^{\z}$ of $\mathbb{D}_{2k}/(\mathbb{D}_{2k})_{\mathbb{D}_{2k}}^{\perp}$, we again conclude that $f\in D^{2k-1}(\mathbb{S}_{2-2k})$ by Corollary \ref{cor:Dperp}.
 The equivalence between \eqref{eqn:innerE} and \eqref{eqn:innerEcoeff} then yields the claim.
\end{proof}

We are now ready to prove the following formal version of Corollary \ref{cor:HeckeExistIntro}.
\begin{corollary}\label{cor:HeckeExist}
The quotient spaces $\mathbb{S}_{2k}^{\z}/\mathbb{S}_{2k}^{\z,\perp}$ and $\mathbb{D}_{2k}/(\mathbb{S}_{2k}^{\z})_{\mathbb{D}_{2k}}^{\perp}$ are finite-dimensional. 
\end{corollary}
\begin{proof}
We note that we have the orthogonal splittings 
\begin{align*}
\mathbb{S}_{2k}^{\z} &=S_{2k}\perp \left(\mathbb{E}_{2k}^{\z} \oplus \mathbb{D}_{2k}^{\z}\Big/\left(\mathbb{S}_{2k}^{\z}\right)_{\mathbb{D}_{2k}^{\z}}^{\perp}\right)\perp \left(\mathbb{S}_{2k}^{\z}\right)_{\mathbb{D}_{2k}^{\z}}^{\perp},
\\
\mathbb{S}_{2k}^{\z} &= S_{2k}\perp \left(\mathbb{E}_{2k}^{\z} \oplus \mathbb{D}_{2k}^{\z}\right)\Big/\left(\mathbb{S}_{2k}^{\z}\right)_{\mathbb{S}_{2k}^{\z}}^{\perp}\perp \left(\mathbb{S}_{2k}^{\z}\right)_{\mathbb{S}_{2k}^{\z}}^{\perp}.
\end{align*}
Since 
$
(\mathbb{S}_{2k}^{\z})_{\mathbb{D}_{2k}^{\z}}^{\perp}\subseteq (\mathbb{S}_{2k}^{\z})_{\mathbb{S}_{2k}^{\z}}^{\perp}, 
$
we have 
\[
\dim_{\C}\left( \left(\mathbb{E}_{2k}^{\z} \oplus \mathbb{D}_{2k}^{\z}\right)\Big/\left(\mathbb{S}_{2k}^{\z}\right)_{\mathbb{S}_{2k}^{\z}}^{\perp}\right)\leq \dim_{\C}\left(\mathbb{E}_{2k}^{\z} \oplus \mathbb{D}_{2k}^{\z}\Big/\left(\mathbb{S}_{2k}^{\z}\right)_{\mathbb{D}_{2k}^{\z}}^{\perp}\right).
\]
To show that $\mathbb{S}_{2k}^{\z}/(\mathbb{S}_{2k}^{\z})_{\mathbb{S}_{2k}^{\z}}^{\perp}$ is finite-dimensional, it therefore suffices to show that 
$
\mathbb{E}_{2k}^{\z} \oplus \mathbb{D}_{2k}^{\z}/(\mathbb{S}_{2k}^{\z})_{\mathbb{D}_{2k}^{\z}}^{\perp}
$
is finite-dimensional. Since 
\[
\mathbb{E}_{2k}^{\z}=\bigoplus_{n=1-2k}^{-1} \C\Psi_{2k,n}^{\z},
\]
we see directly that $\mathbb{E}_{2k}^{\z}$ is finite-dimensional and we only need to show that \linebreak 
$
\mathbb{D}_{2k}^{\z}/(\mathbb{S}_{2k}^{\z})_{\mathbb{D}_{2k}^{\z}}^{\perp}
$
is finite-dimensional. By Corollary \ref{cor:Hecke} (1), we have
\begin{align*}
\left(\mathbb{S}_{2k}^{\z}\right)_{\mathbb{D}_{2k}^{\z}}^{\perp}&=\left\{ f\in D^{2k-1}\left(\mathbb{S}_{2-2k}^{\z}\right): f=D^{2k-1}(F),\ F\in\mathbb{S}_{2-2k}^{\varrho}, p_{\mathfrak{F}_{2-2k}(F),\z}(z)=0\right\}=:\mathbb{J}_{2k,\z}^{\z}.
\end{align*}
By taking the intermediary quotient with $(\mathbb{D}_{2k}^{\z})_{\mathbb{D}_{2k}^{\z}}^{\perp}$, we have 
\[
\dim_{\C}\left(\mathbb{D}_{2k}^{\z}\Big/\mathbb{J}_{2k,\z}^{\z}\right)= \dim_{\C}\left( \mathbb{D}_{2k}^{\z}\Big/\left(\mathbb{D}_{2k}^{\z}\right)_{\mathbb{D}_{2k}^{\z}}^{\perp}\right)+ \dim_{\C}\left(\left(\mathbb{D}_{2k}^{\z}\right)_{\mathbb{D}_{2k}^{\z}}^{\perp}\Big/\mathbb{J}_{2k,\z}^{\z}\right).
\]
Since 
$
(\mathbb{D}_{2k})_{\mathbb{D}_{2k}}^{\perp}\cap \mathbb{D}_{2k}^{\z}\subseteq (\mathbb{D}_{2k}^{\z})_{\mathbb{D}_{2k}^{\z}}^{\perp},
$
we have 
\[
\dim_{\C}\left( \mathbb{D}_{2k}^{\z}\Big/\left(\mathbb{D}_{2k}^{\z}\right)_{\mathbb{D}_{2k}^{\z}}^{\perp}\right)\leq  \dim_{\C}\left( \mathbb{D}_{2k}\Big/\left(\mathbb{D}_{2k}\right)_{\mathbb{D}_{2k}}^{\perp}\right)=\dim_{\C}\left(S_{2k}\right)
\]
is finite (and indeed, the proof of Corollary \ref{cor:Hecke} (1) implies that the dimensions agree).

Finally, as shown in Corollary \ref{cor:Hecke} (1), 
$
(\mathbb{D}_{2k}^{\z})_{\mathbb{D}_{2k}^{\z}}^{\perp}=D^{2k-1}(\mathbb{S}_{2-2k}^{\z})
$
and we see that 
\[
 \dim_{\C}\left(\left(\mathbb{D}_{2k}^{\z}\right)_{\mathbb{D}_{2k}^{\z}}^{\perp}\Big/\mathbb{J}_{2k,\z}^{\z}\right)= \dim_{\C}\left(D^{2k-1}\left(\mathbb{S}_{2-2k}^{\z}\right)\Big/\mathbb{J}_{2k,\z}^{\z}\right)\leq 2k-1
\]
because the polynomial part is a polynomial of degree at most $2k-2$ and hence there are at most $2k-1$ linearly-independent possible polynomial parts. 
We therefore conclude that $\mathbb{S}_{2k}^{\z}/(\mathbb{S}_{2k}^{\z})_{\mathbb{S}_{2k}^{\z}}^{\perp}$ is finite-dimensional. 

\noindent The argument that $\mathbb{D}_{2k}/\left(\mathbb{S}_{2k}^{\z}\right)_{\mathbb{D}_{2k}}^{\perp}$ is finite-dimensional is similar. By Corollary \ref{cor:Hecke} (2)
\begin{align*}
\left(\mathbb{S}_{2k}^{\z}\right)_{\mathbb{D}_{2k}}^{\perp}&= \left\{ f\in D^{2k-1}\left(\mathbb{S}_{2-2k}\right): f=D^{2k-1}(F),\ F\in\mathbb{S}_{2-2k}, p_{\mathfrak{F}_{2-2k}(F),\z}(z)=0\right\}=:\mathbb{J}_{2k,\z}.
\end{align*}
Thus we want to show that 
$
\mathbb{D}_{2k}/\mathbb{J}_{2k,\z} 
$
is finite-dimensional. By taking the intermediary quotient with $(\mathbb{D}_{2k})_{\mathbb{D}_{2k}}^{\perp}=D^{2k-1}(\mathbb{S}_{2-2k})$ (using Corollary \ref{cor:Dperp}), we have
\[
\dim_{\C}\left(\mathbb{D}_{2k}\Big/\mathbb{J}_{2k,\z}\right) =\dim_{\C}\left( \mathbb{D}_{2k}/\left(\mathbb{D}_{2k}\right)_{\mathbb{D}_{2k}}^{\perp}\right) + \dim_{\C}\left( D^{2k-1}\left(\mathbb{S}_{2-2k}\right)\Big/\mathbb{J}_{2k,\z}\right).
\]
By Corollary \ref{cor:DperpFD}, we have 
\[
\dim_{\C}\left( \mathbb{D}_{2k}\Big/\left(\mathbb{D}_{2k}\right)_{\mathbb{D}_{2k}}^{\perp}\right)=\dim_{\C}\left(S_{2k}\right)<\infty,
\qquad
\dim_{\C}\left( D^{2k-1}\left(\mathbb{S}_{2-2k}\right)\Big/\mathbb{J}_{2k,\z}\right)\leq 2k-1
\]
because the polynomial part has at most $2k-1$ linearly-independent choices. 
\end{proof}
We next investigate to what extent the restriction $\xi_{2-2k}(F),\xi_{2-2k}(G)\in \mathbb{D}_{2k}\perp S_{2k}$ is necessary in Theorem \ref{thm:xiDinner}. We begin by reducing the first statement of Proposition \ref{prop:EperpIntro} to the second statement in Proposition \ref{prop:EperpIntro}.

\begin{lemma}\label{lem:NoExtend}
If Theorem \ref{thm:xiDinner} were to hold for arbitrary $F,G\in\mathscr{H}_{2-2k}$ with $\xi_{2-2k}(F)\in \mathbb{E}_{2k}$ and $\xi_{2-2k}(G)\in S_{2k}$, then $\mathbb{D}_{2k}$ would be orthogonal to $\mathbb{E}_{2k}$. 
\end{lemma}
\begin{proof}
By Lemma \ref{lem:DXicusp} (3), for every $F\in \mathscr{H}_{2-2k}^{\mathbb{E}}$ we have $\xi_{2-2k}(F)\in \mathbb{E}_{2k}$, and by Lemma \ref{lem:DXicusp} (1) for every $G\in \mathscr{H}_{2-2k}^{\operatorname{cusp}}$ we have $\xi_{2-2k}(G)\in S_{2k}$. Since the spaces $\mathbb{E}_{2k}$ and $S_{2k}$ are orthogonal by definition, the extension of Theorem \ref{thm:xiDinner} implies that 
\begin{equation}\label{eqn:flippedE}
0=\left<\xi_{2-2k}(F),\xi_{2-2k}(G)\right> =-\frac{(4\pi)^{4k-2}}{(2k-2)!^2} \left<D^{2k-1}(G),D^{2k-1}(F)\right>.
\end{equation}
Let $f\in \mathbb{E}_{2k}$ and $g\in \mathbb{D}_{2k}$ be given. By the surjectivity in Lemma \ref{lem:DXicusp} (3), there exists $F\in \mathscr{H}_{2-2k}^{\mathbb{E}}$ such that $D^{2k-1}(F)=f$, while the surjectivity in Lemma \ref{lem:DXicusp} (2) implies that there exists $G\in \mathscr{H}_{2-2k}^{\operatorname{cusp}}$ with $D^{2k-1}(G)=g$. Plugging this into \eqref{eqn:flippedE}, we have 
\[
\left<f,g\right>=\left<D^{2k-1}(G),D^{2k-1}(F)\right> = \frac{(2k-2)!^2}{(4\pi)^{4k-2}} \left<\xi_{2-2k}(F),\xi_{2-2k}(G)\right> =0.\qedhere
\]
\end{proof}
We finally compute the space orthogonal to all of $\mathbb{E}_{2k}$, yielding the second claim in Proposition \ref{prop:EperpIntro}.
\begin{proposition}\label{prop:Eperp}
We have
\[
\left(\mathbb{E}_{2k}\right)_{\mathbb{S}_{2k}}^{\perp}=S_{2k}.
\]
In particular, 
\[
\left(\mathbb{E}_{2k}\right)_{\mathbb{D}_{2k}}^{\perp}=\left(\mathbb{E}_{2k}\right)_{\mathbb{E}_{2k}}^{\perp}=\mathbb{S}_{2k}^{\perp}=\{0\},
\]
 and the spaces $\mathbb{E}_{2k}$,  $\mathbb{E}_{2k}\oplus \mathbb{D}_{2k}$, and $\mathbb{S}_{2k}$ are hence non-degenerate.
\end{proposition}
\begin{proof}
Suppose for contradiction that $f\in \mathbb{E}_{2k}\oplus \mathbb{D}_{2k}$ is orthogonal to $\mathbb{E}_{2k}$ and let $F$ be given such that $\xi_{2-2k}(F)=f$.  By Lemma \ref{lem:innermero} with $\ell=-1$, we have 
\[
c_{F,\z}^+(0)=0
\]
for every $\z\in\H$. Let $\z$ be any point at which $F$ does not have a singularity. Then taking $z\to \z$ in the elliptic expansion \eqref{eqn:elliptic} and using \cite[Lemma 5.4]{BKRohrlich},  we have 
\[
F(\z)=(\z-\overline{\z})^{2k-2} c_{F,\z}^+(0) = 0.
\]
We conclude that $F(\z)=0$ at $\z\in\H$ where $F$ does not have a singularity. Then $f=\xi_{2-2k}(F)$ satisfies $f(\z)=0$ for all $\z$ where $F$ does not have a pole. Since $f$ is meromorphic, we conclude that $f=0$. Therefore 
$
\left(\mathbb{E}_{2k}\right)_{\mathbb{S}_{2k}}^{\perp}=S_{2k}. 
$
Finally, if $f\in \mathbb{S}_{2k}^{\perp}$, then in particular $f\in (\mathbb{E}_{2k})_{\mathbb{S}_{2k}}^{\perp}=S_{2k}$. Since $S_{2k}$ is positive-definite, $\mathbb{S}_{2k}^{\perp}=\{0\}$. 
\end{proof}


\begin{thebibliography}{99}

\bibitem{Borcherds} R. Borcherds, \begin{it}Automorphic forms with singularities on Grassmannians\end{it}, Invent. Math. \textbf{132} (1998), 491--562.

\bibitem{BDE} K. Bringmann, N. Diamantis, and S. Ehlen, \begin{it}Regularized inner products and errors of modularity\end{it}, Int. Math. Res. Not. \textbf{2017} (2017), 7420--7458.
\bibitem{Book}K. Bringmann, A. Folsom, K. Ono, and L. Rolen, \begin{it}Harmonic Maass forms and mock modular forms:  theory and applications\end{it}, AMS Colloquium Series (2017).

\bibitem{BJK} K. Bringmann, P. Jenkins, and B. Kane, \begin{it} Differential operators on polar harmonic Maass forms and elliptic duality\end{it}, Q. J. Math. \textbf{70} (2019), 1181--1207.

\bibitem{BKRohrlich} K. Bringmann and B. Kane, \begin{it}An extension of Rohrlich's Theorem to the $j$-function\end{it}, Forum Math. Sigma \textbf{8} (2019), 1--33.



\bibitem{BKvP} K. Bringmann, B. Kane, and A. von Pippich, \begin{it}Regularized inner products of meromorphic modular forms and higher Green's Functions\end{it}, Communications in Contemporary Mathematics \textbf{21} (2019), 1850029.
















\bibitem{HM} J. Harvey and G. Moore, \begin{it}Algebras, BPS states, and strings\end{it}, Nuclear Phys. B \textbf{463} (1996), 315--368.















\bibitem{PeMet} H. Petersson, \begin{it}\"Uber eine Metrisierung der ganzen Modulformen\end{it}, Jahresbericht der Deutschen Mathematiker-Vereinigung \textbf{49} (1939), 49--75.

\bibitem{PeEinheit} H. Petersson, \begin{it}Einheitliche Begr\"undung der Vollst\"andigkeitss\"atze f\"ur die Poincar\'eschen Reihen von reeller Dimension bei beliebigen Grenzkreisgruppen von erster Art\end{it} Abh. Math. Sem. Univ. Hmbg. \textbf{14} (1941), 22--60.








\bibitem{Pe2} H. Petersson, \begin{it}\"Uber automorphe Orthogonalfunktionen und die Konstruktion der automorphen Formen von positiver reeller Dimension\end{it}, Math. Ann. \textbf{127} (1954), 33--81.






\bibitem{ZagierNotRapid}D. Zagier, \begin{it} The Rankin--Selberg method for automorphic functions which are not of rapid decay\end{it}, J. Fac. Sci Tokyo \textbf{28} (1982), 415--438.
\end{thebibliography}
\end{document}